\documentclass[a4paper,reqno, 11pt]{amsart}
\usepackage{fouriernc} 

\usepackage{amsmath}
\usepackage{amsfonts}
\usepackage{amssymb}
\usepackage{amsthm}
\usepackage{amscd}
\usepackage{MnSymbol}
\usepackage{enumerate}
\usepackage{color}
\usepackage{hyperref} 
\usepackage[cmtip,all]{xy}
\usepackage{enumitem}

\hypersetup{
  linktoc=page
}

\setitemize{leftmargin=0.75cm}

\newcommand{\N}{\mathbf{N}}				
\newcommand{\Z}{\mathbf{Z}}				
\newcommand{\Q}{\mathbf{Q}}				
\newcommand{\F}{\mathbf{F}}				

\newcommand{\p}{\mathfrak{p}}			
\renewcommand{\O}{\mathcal{O}}			

\newcommand{\nach}[1][]{\overset{#1}{\rightarrow}}		
\newcommand{\Nach}[1][]{\overset{#1}{\longrightarrow}}	
\newcommand{\abb}{\mapsto}				
\newcommand{\inj}{\hookrightarrow}		
\newcommand{\surj}{\twoheadrightarrow}	
\newcommand{\iso}{\Nach[\cong]}	 

\newcommand{\inv}{^{-1}}								
\newcommand{\sd}{\, | \,}								
\newcommand{\Sd}{\, \big| \,}							
\newcommand{\skp}[1]{\langle #1\rangle}				
\newcommand{\nt}{\vartriangleleft}						

\DeclareMathOperator{\Ho}{H}				
\DeclareMathOperator{\B}{B}					
\DeclareMathOperator{\id}{id}				
\DeclareMathOperator{\GL}{GL}				
\DeclareMathOperator{\Tor}{Tor}			
\DeclareMathOperator{\Frac}{Frac}			

\DeclareMathOperator{\K}{K}							
\newcommand{\KM}{\K^\mathrm{M}}					
\newcommand{\iKM}{\hat{\K}^\mathrm{M}}				
\newcommand{\PhiMQ}{\Phi_{\mathrm{MQ}}}			
\newcommand{\PhiQM}{\Phi_{\mathrm{QM}}}			
\DeclareMathOperator{\T}{T}							
\DeclareMathOperator{\StR}{StR}					
\DeclareMathOperator{\BGL}{BGL} 					
\newcommand{\Mtame}{\partial}			
\newcommand{\iMtame}{\hat{\partial}}	

\newcommand{\df}[1]{\textsc{#1}}					
\newcommand{\etale}{\'{e}tale}    					

\renewcommand{\phi}{\varphi}
\renewcommand{\epsilon}{\varepsilon}
\renewcommand{\theta}{\vartheta}
\renewcommand{\rho}{\varrho}

\theoremstyle{definition}				
	\newtheorem{defin}{Definition}[section]
	
\theoremstyle{plain} 					
	\newtheorem{thm}[defin]{Theorem}
	\newtheorem*{thm*}{Theorem}
	\newtheorem{prop}[defin]{Proposition}
	\newtheorem{lemma}[defin]{Lemma}
	\newtheorem{cor}[defin]{Corollary}
\theoremstyle{remark} 					
	\newtheorem*{remark}{Remark}



\newcounter{formulanumber}[defin]

\newcommand{\vquad}{\vspace{6pt}}

\title[Milnor K-theory of complete DVR's with finite residue fields]{Milnor K-theory of complete discrete valuation rings with finite residue fields}
\author{Christian Dahlhausen}
\address{Fakult\"at f\"ur Mathematik, Universit\"at Regensburg, 93040 Regensburg, Germany}
\email{christian.dahlhausen(at)ur.de}
\thanks{The author is supported by the DFG GRK 1692 ``Curvature, Cycles and Cohomology''}

\begin{document}

\begin{abstract}
Consider a complete discrete valuation ring $\O$ with quotient field~$F$ and finite residue field. Then the inclusion map $\O \inj F$ induces a map $\iKM_*\O \nach \iKM_*F$ on improved Milnor K-theory. We show that this map is an isomorphism in degrees bigger or equal to 3. This implies the Gersten conjecture for improved Milnor \mbox{K-theory} for $\O$. This result is new in the $p$-adic case. 
\end{abstract}

\maketitle


	\section{Introduction}

Let $\iKM_*$ denote the improved Milnor K-theory as introduced by \df{Gabber} \cite{gabber} and mainly developed by \df{Kerz} \cite{kerz}. For fields, it coincides with the usual Milnor K-theory. For $n\geq 1$ and a discrete valuation ring $\O$ with quotient field $F$ and residue field $\kappa$, there is the so-called Gersten complex
	\begin{align*}
		0 \Nach \iKM_n\O \Nach \iKM_nF \Nach \iKM_{n-1}\kappa \Nach 0.
	\end{align*}
This complex is known to be right-exact. Is it also exact on the left side? In the equicharacteristic case this was shown by \df{Kerz} \cite[Prop.~10~(8)]{kerz}.
 As a slight extension to a special case of the mixed characteristic case, we show the exactness on the left side by proving the following result (Theorem~\ref{thm_main_result}).

\begin{thm*}
Let $\O$ be a complete discrete valuation ring with quotient field $F$ and finite residue field. Then for $n\geq 3$ the inclusion map $\iota\,\colon\,\O\inj F$ induces an isomorphism
	\begin{align*}
		\iota_* \,\colon \, \iKM_n\O \iso \iKM_n F
	\end{align*}
on improved Milnor K-theory.
\end{thm*}	

In the \mbox{$p$-adic} case this is new. We prove the theorem as follows: We show that $\iKM_n\O$ is a divisible abelian group for $n\geq 3$ (Theorem \ref{KM_n_O_divisible}) using results from the appendix of \df{Milnor}'s book \cite{milnor}. Combining this with the unique divisibility of $\iKM_nF$ (proved by \df{Sivitskii} \cite{sivitskii}) and a comparision between algebraic and Milnor K-theory (proved by \df{Nesterenko} and \df{Suslin} \cite{suslin}) yields the theorem.

\emph{Acknowledgements.} I am grateful to my advisor Moritz Kerz for proposing me this interesting topic. Also, I thank Morten L\"uders for useful comments and fruitful discussions and Thomas Fiore for having a look at the language. I want to thank my unknown referee from AKT (though my submission was finally declined). Their reports helped a lot in clarifying the presentation of this paper's content.  Last, but not least, I thank Chuck Weibel for helpful comments on the presentation of the paper's content.

		\section{Milnor K-theory}
		
\begin{defin} \label{defin_milnor_k_fields}
Let $A$ be a commutative ring with unit, $\T_* A^\times$ be the (non-commutative) tensor algebra of $A^\times$ over $\Z$,  and $\StR_* A^\times$ the homogeneous ideal of $\T_* A^\times$ which is generated by the set $\{x\otimes y \in\T_2A^\times \sd x+y=1\}$ (the so-called \df{Steinberg relations}). Define the \df{Milnor K-theory} of $A$ to be the graded ring
	\begin{align*} 
		\KM_* A := \T_* A^\times/\StR_* A^\times.
	\end{align*} 
For $x_1,\ldots,x_n\in A^\times$ let $\{x_1,\ldots,x_n\}$ denote the image of $x_1\otimes\ldots\otimes x_n$ under the natural homomorphism $\T_nA^\times \nach \KM_nA$. Evidently, this yields a functor from commutative rings to abelian groups.
\end{defin}	

This notion behaves well if $A$ is a field or a local ring with infinite residue field \cite{kerz_gersten_conjecture}. But some nice properties do not hold for arbitrary commutative rings (e.g. that the natural map to algebraic K-theory is an isomorphism in degree 2). For local rings, this lack is repaired by a generalisation due to \df{Gabber} \cite{gabber}, the improved Milnor K-theory, which was mainly developed by \df{Kerz} \cite{kerz}.

\begin{defin} \label{def_rational_functions}
Let $A$ be a local ring and $n \in \N := \Z_{\geq0}$. The subset
		\begin{align*}
		S := \bigl\{ \sum_{\underline{i}\in \N^n} a_{\underline{i}}\cdot \underline{t}^{\underline{i}} \in A[t_1,\ldots,t_n] \,\Sd\, \skp{a_{\underline{i}}\sd\underline{i}\in \N^n} = A \bigr\}  
		\end{align*}
of $A[t_1,\ldots,t_n]$ is multiplicatively closed, where $\underline{t}^{\underline{i}} = t_1^{i_1}\cdot\ldots\cdot t_n^{i_n}$. Define the \df{ring of rational functions} (in $n$ variables) to be $A(t_1,\ldots,t_n) := S\inv A[t_1,\ldots,t_n]$. We obtain maps $\iota \colon A \nach A(t)$ and $\iota_1,\iota_2 \colon A(t) \nach A(t_1,t_2)$ by mapping $t$ respectively to $t_1$ or $t_2$.

For $n\geq 0$ we define the $n$\df{-th improved Milnor K-theory} of $A$ to be
	\begin{align*} 
	\iKM_nA := \ker\bigl[\KM_nA(t) \Nach[\delta^\mathrm{M}_n] \KM_nA(t_1,t_2) \bigr],
	\end{align*} 
where $\delta^\mathrm{M}_n := \KM_n(\iota_1) - \KM_n(\iota_2)$. By definition, we have an exact sequence
	\begin{align*}  \label{improved_exact_sequence}
	0 \Nach \iKM_nA \Nach[\iota_*] \KM_nA(t) \Nach[\delta^\mathrm{M}_n] \KM_nA(t_1,t_2).
	\end{align*} 
In particular, for $n=0$ we have $\iKM_0 = \ker\bigl(\Z\Nach[0]\Z\bigr) = \Z$. 
As a direct consequence of the construction we obtain a natural homomorphism  
	\begin{align*} 
		\KM_*A \Nach \iKM_*A.
	\end{align*}
\end{defin}

We state some facts about Milnor K-theory of local rings. 
\begin{prop} \label{iKM_properties}
Let $A$ be a local ring. Then:
\begin{enumerate}	
	\item $\iKM_1A \cong A^\times$.
	\item For $\alpha \in \iKM_nA$ and $\beta \in \iKM_mA$ we have $\alpha\beta = (-1)^{nm} \beta\alpha$, i.e. $\iKM_*A$ is graded-commutative.	
	\item For $x\in A^\times$ the equations $\{x,-x\} = 0$ and $\{x,x\} = \{x,-1\}$ hold in $\iKM_2A$.
	\item There exists a natural homomorphism 
		\begin{align*} 
			\PhiMQ(A) \,\colon\, \iKM_*A \Nach \K_*A.
		\end{align*}
		where $\K_*$ denotes algebraic K-theory due to \df{Quillen}. In degree 2, this is   		an isomorphism.
	\item If $A$ is a field, the natural homomorphism
		\begin{align*} 
			\KM_*A \Nach \iKM_*A
		\end{align*} 
		is an isomorphism.
	\item If $A$ is a finite field, then $\iKM_nA \cong 0$ for $n\geq 2$.
\end{enumerate}	
\end{prop}
\begin{proof}
Everything is due to \df{Kerz} and references within this proof refer to \cite{kerz} (unless stated otherwise). (i) and (ii) is [Prop.~10~(1), (2)]. The first statement of (iv) follows from [Thm.~7] and the existence of a natural map $\KM_*A \to \K_*A$. The second statement of (iv) and (iii) is [Prop.~10~(3)] combined with [Prop.~2]. (v) is [Prop.~10~(4)]. (vi) follows from (v) and \cite[Ex.~1.5]{milnor}.
Precisely, (ii) and (iii) rely on the corresponding facts for $\KM_*$; for proofs of them see e.g. \cite[Prop.~7.1.1, Lem.~7.1.2]{gille}.
\end{proof}

\begin{thm}[{\cite[Thm.~A, Thm.~B]{kerz}}]
Let $A$ be a local ring. Then:
\begin{enumerate}
	\item The natural homomorphism
	\begin{align*}
		\KM_*A \Nach \iKM_*A
	\end{align*}
is surjective.
	\item If $A$ has infinite residue field, the homomorphism
		\begin{align*} 
			\KM_*A \Nach \iKM_*A
		\end{align*} 
		is an isomorphism.
\end{enumerate}
\end{thm}

\begin{lemma} \label{KM_DVR_generators}
Let $\O$ be a discrete valuation ring with quotient field $F$ and let $\pi\in\O$ be a uniformising element. For $n\geq 1$ we have
	\begin{align*}
		\KM_nF = \bigl< \{\pi,u_2,\ldots,u_n\},\{u_1,\ldots,u_n\} \Sd u_1,\ldots,u_n \in \O^\times \bigr>.
	\end{align*}
\end{lemma}
\begin{proof}
This follows straightforwardly from Proposition~\ref{iKM_properties}~(iii) and the fact that every element $x\in F^\times$ has a description $x=u\pi^k$ for suitable $u\in\O^\times$ and $k\in\Z$.
\end{proof}

\begin{thm}[{\cite[Lem.~2.1]{milnor_k}, cf. \cite[Prop.~7.1.4]{gille}}] \label{tame_symbol}
Let $\O$ be a discrete valuation ring with quotient field $F$, discrete valuation $v\, \colon\, F^\times \nach \Z$,  and residue field~$\kappa$. Then for every $n\geq 1$ there exists a unique homomorphism (the \df{tame symbol})
	\begin{align*} 
	\partial \,\colon\, \KM_nF \Nach \KM_{n-1}\kappa
	\end{align*} 
such that for every uniformising element $\pi\in\O$ and all units $u_2,\ldots,u_n\in\O^\times$ we have
	\begin{align*} 
	\partial \bigl( \{\pi,u_2,\ldots,u_n\} \bigr) = \{\bar{u}_2,\ldots,\bar{u}_n\}.
	\end{align*} 
We also write $\partial_\pi$ or $\partial_v$ to indicate the considered valuation.
\end{thm}

\begin{prop} \label{tame_symol_exact_sequence}
Let $\O$ be a discrete valuation ring with quotient field $F$ and residue field $\kappa$. For $n\geq 1$ we have an exact sequence of groups
	\begin{align*} 
		0 \Nach V_n \Nach \KM_nF \Nach[\partial] \KM_{n-1}\kappa \Nach 0,
	\end{align*} 
where $V_n$ is the subgroup of $\KM_nF$ generated by 
	\begin{align*}
		\bigl\{ \{u_1,\ldots,u_n\} \Sd u_1,\ldots,u_n\in\O^\times \bigr\}.
	\end{align*}
As a consequence, we have an exact sequence
	\begin{align*} 
		\iKM_n\O \Nach[\iota_*] \iKM_nF \Nach[\iMtame] \iKM_{n-1}\kappa \Nach 0
	\end{align*} 
where $\iMtame$ is induced by $\Mtame$ and the natural isomorphism of Proposition~\ref{iKM_properties}~(v).
\end{prop}
\begin{proof}
For the exactness of the first sequence look at \cite[Prop.~1.7.1]{gille}. The exactness of the second sequence is shown by a diagram chase in the commutative diagram
	\[ \begin{xy}
	\xymatrix{
		\KM_n\O \ar[r]^{\iota_*} \ar@{->>}[d]
		& \KM_nF \ar[r]^\Mtame \ar[d]^-\cong
		& \KM_{n-1}\kappa \ar[r] \ar[d]^-\cong 
		& 0 \\
		\iKM_n\O \ar[r]^{\iota_*}
		& \iKM_nF \ar[r]^\iMtame
		& \iKM_{n-1}\kappa \ar[r]
		& 0	. 
	} 
	\end{xy}  \] 
\end{proof}

\begin{prop}[{\cite[Thm.~1.3 a)]{gersten}, cf. \cite[IV.~Cor.~1.13; V.~Cor.~6.9.2]{weibel}}] \label{gersten_conjecture}
Let $\O$ be a discrete valuation ring with quotient field $F$ and \emph{finite} residue field $\kappa$. Then for all $n\geq 0$ there is an exact sequence
	\begin{align*}
		0 \Nach \K_n\O \Nach \K_nF \Nach \K_{n-1}\kappa \Nach 0.
	\end{align*}
Furthermore, $\K_{2n}\kappa=0$ for $n\geq 1$ \cite{quillen}.
\end{prop}
In particular, by using Proposition~\ref{iKM_properties}~(iv), there is a short exact sequence
	\begin{align} \label{gersten_sequence_milnor_2}
		0 \Nach \iKM_2\O \Nach \iKM_2F \Nach \iKM_1\kappa \Nach 0.
	\end{align}

The goal is to generalise the sequence \eqref{gersten_sequence_milnor_2} to arbitrary $n\geq 2$. We want to know if for any discrete valuation ring $\O$ and all $n\geq 2$ the sequences
	\begin{align}  \label{gersten_sequenz}
		0 \Nach \iKM_n\O \Nach[\iota_*] \iKM_nF \Nach[\Mtame] \iKM_{n-1}\kappa \Nach 0
	\end{align}
are exact? For algebraic K-theory, \df{Gersten} conjectured that the analogous sequences to \eqref{gersten_sequenz} are exact for all $n\geq 2$ \cite{gersten}. Thus we refer to this question as the Gersten conjecture for improved Milnor K-theory. In the case that $\O$ is complete and $\kappa$ is finite this will be an immediate consequence of our main result which relies on an analogous statement by \df{Nesterenko} and \df{Suslin} for local rings with infinite residue field and classical Milnor K-theory \cite[Thm.~4.1]{suslin}. A proof based on their result will be presented in the appendix.

\begin{prop}[{\cite[Prop.~10~(6)]{kerz}}] \label{MQM}
Let $A$ be a local ring and $n\in\N$. Then there exists a map
	\begin{align*}
		\PhiMQ(A) \,\colon\, \iKM_nA \Nach \K_nA
	\end{align*}
as well as a map
	\begin{align*}
		\PhiQM(A) \,\colon\, \K_nA \Nach \iKM_nA
	\end{align*}
such that the composition
	\begin{align*}
		\iKM_nA \Nach[\PhiMQ(A)] \K_nA \Nach[\PhiQM(A)] \iKM_nA
	\end{align*}
is multiplication with $\chi_n := (-1)^{n-1}\cdot (n-1)!$.
\end{prop}

	\section{Divisibility of $\iKM_n\O$ for $n\geq 3$}
			\label{section_div_of_iKM}
		
In this section we prove that $\iKM_n\O$ is divisible for a complete discrete valuation ring $\O$ with finite residue field and $n\geq 3$. This result will be the key ingredient for the proof of our main result. First, we examine the divisibility prime to $p$.

\begin{defin}
For an abelian group $A$ and $m\in\Z$ we set $A/m := A/mA$, where $mA = \{ma \sd a\in A\}$. Thus $A/m\cong A\otimes_\Z\Z/m$.
\end{defin}

\begin{lemma} \label{mod_m}
Let $\O$ be a complete discrete valuation ring with quotient field~$F$ and finite residue field $\kappa$ of characteristic $p$. Furthermore, let $m\in\Z$ be such that $(p,m)=1$. Then the canonical projection $\O \surj \kappa$ induces an isomorphism
	\begin{align*}
		\bigl(\KM_*\O\bigr)/m \iso \bigl(\KM_*\kappa\bigr)/m
	\end{align*}
of graded rings.
\end{lemma}
\begin{proof}
The claim is trivial for $\KM_0$. From Hensel's Lemma we obtain
	\begin{align*}
	\O^\times/m \cong (U_1 \times \kappa^\times)/m \cong (U_1/m) \times (\kappa^\times/m) \cong \kappa^\times/m.
	\end{align*}
This implies the claim for $\KM_1$. Now let $n\geq 2$. Then we have the following commutative diagram with exact rows
	\begin{align} \label{diag1}
	\begin{xy}  
	\xymatrix{
	0 \ar[r] & \StR_n\O^\times \ar[r] \ar[d] & \T_n\O^\times \ar[r] \ar[d] & \KM_n\O \ar[r] \ar[d] & 0  \\
	0 \ar[r] & \StR_n\kappa^\times \ar[r] & \T_n\kappa^\times \ar[r] & \KM_n\kappa \ar[r] & 0,
	}  
	\end{xy} 
	\end{align}
where $\T_*\O^\times$ and $\T_*\kappa^\times$ denote the tensor algebras of $\O^\times$ and $\kappa^\times$ and
	\begin{align*} 
	\StR_n\O^\times & = \skp{x_1\otimes\ldots\otimes x_n \in \T_n\O^\times \sd \exists  i\neq j \colon x_i+x_j=1} \\
	\StR_n\kappa^\times & = \skp{\bar{x}_1\otimes\ldots\otimes\bar{x}_n \in \T_n\kappa^\times\sd \exists  i\neq j \colon \bar{x}_i+\bar{x}_j=\bar{1}}.
	\end{align*}
Note that $\overline{y}=\overline{1}$ in $\kappa$ implies $y\in U_1$. Tensoring the diagram \eqref{diag1} with $\Z/m$ we obtain the following commutative diagram with exact lines
	\begin{align} \label{diag2}
	 \begin{xy} 
	\xymatrix{
	\StR_n\O^\times/m \ar[r] \ar[d]^\alpha & \T_n\O^\times/m \ar[r] \ar[d]^\beta 
		& \KM_n\O/m \ar[r] \ar[d]^\gamma & 0  \\
	\StR_n\kappa^\times/m \ar[r] & \T_n\kappa^\times/m \ar[r] 
		& \KM_n\kappa/m \ar[r] & 0.
	}  
	\end{xy}
	\end{align}
It is easy to check that $(\T_*A)/m\cong\T_*(A/m)$ for an arbitrary abelian group $A$; along with $\O^\times/m\cong\kappa^\times/m$ we see that $\beta$ is an isomorphism.

Now let's show that $\alpha$ is surjective. Therefore let $n\in\N$ and $\bar{x}_1 \otimes\ldots\otimes \bar{x}_n \in\StR_n\kappa^\times$. Without loss of generality we can assume that $\bar{x}_1+\bar{x}_2=\bar{1}$, i.e. $x_1+x_2=:u\in U_1$, hence $u\inv x_1+u\inv x_2=1$ in $\O$. Thus $\xi:= u\inv x_1\otimes u\inv x_2\otimes x_3\otimes\ldots\otimes x_n\in\StR_n\O^\times$ satisfies $\alpha(\xi) = \bar{x}_1 \otimes\ldots\otimes \bar{x}_n$. So $\alpha$ is surjective.

Applying the five lemma to \eqref{diag2} demonstrates that $\gamma$ is an isomorphism. This concludes the proof.
\end{proof}

\begin{cor} \label{divisible_coprime}
Let $\O$ be a complete discrete valuation ring with quotient field $F$ and finite residue field $\kappa$ of characteristic $p$. Then for $n\geq 2$ the groups $\KM_n\O$ and $\iKM_n\O$ are $m$-divisible if $(p,m)=1$.
\end{cor}
\begin{proof}
For $n\geq 2$ and $m\in\Z$ such that $(p,m)=1$ we have $(\KM_n\O)/m \cong (\KM_n\kappa)/m = 0$, whence $\KM_n\kappa=0$ since $\kappa$ is a finite field (see Proposition~\ref{iKM_properties}~(iv)). Thus $\KM_n\O$ is $m$-divisible. The surjective homomorphism $\KM_n\O \surj \iKM_n\O$ shows the claim for $\iKM_n\O$.
\end{proof}

Next we want to understand multiplication by $p$ on Milnor K-theory. For that we need the following results.

\begin{prop} \label{delta}
Let $\O$ be a complete discrete valuation ring with quotient field $F$ and finite residue field of characteristic~$p$.
	\begin{enumerate}
	\item Let $F$ contain all $p$-th roots of unity. Then $(\iKM_2F)/p$ is cyclic of order $p$.
	\item If $F$ contains only the trivial $p$-th root of unity, $(\iKM_2F)/p$ is zero.
	\end{enumerate}
Here we use tacitly the fact that $\iKM_*F\cong \KM_*F$ from Proposition~2.3~(v).
\end{prop}
\begin{proof}
This is done in \cite{milnor}, (i) is [Cor.~A.12] and (ii) follows from [Lem.~A.6] together with [Lem.~A.4].
\end{proof}

To treat a special case in our proof below, we need a norm map for a non-\etale{} extension which is not covered literally by \df{Kerz}' work. However, everything needed is basically already contained in \cite{kerz_gersten_conjecture}.

\begin{prop} \label{norm}
Let $A$ be a local and factorial domain and $\iota\colon A\inj B$ an extension of local rings such that $B\cong A[X]/\skp{\pi}$ for a monic irreducible polynomial $\pi\in A[X]$.
Then there exists a transfer morphism, the \df{norm map},
	\begin{align*}
		N_{B/A} \,\colon\, \iKM_*B \Nach \iKM_*A
	\end{align*}
satisfying the projection formula
	\begin{align} \label{projection_formula}
		N_{B/A}\bigl(\{\iota_*(x),y\}\bigr) = \bigl\{x,N_{B/A}(y)\bigr\}
	\end{align}
for all $x\in\iKM_*A$ and $y\in\iKM_*B$. 
\end{prop}
\begin{proof}
This will be done in the appendix, see \ref{norm_appendix}.
\end{proof}

\begin{prop}[{Bass, Tate, cf.~\cite[Cor.~A.15]{milnor}}] \label{norm_surjective}
Let $\O \subset \O'$ be complete discrete valuation rings with quotient fields $F \subset F'$ and finite residue fields. Then the norm map
	\begin{align*}
		N_{F'/F} \,\colon\, \iKM_2F' \surj \iKM_2F.
	\end{align*}
is surjective.
\end{prop}

With this we can prove the missing $p$-divisibility.

\begin{thm} \label{KM_n_O_divisible}
Let $\O$ be a complete discrete valuation ring with quotient field $F$, finite residue field $\kappa$ of characteristic $p$, and $n\geq 3$. Then $\iKM_n\O$ is divisible.
\end{thm}
\begin{proof} 
The homomorphism
	\begin{align*}
		\KM_3\O \otimes \KM_{n-3}\O &\Nach \KM_n\O \\
		\{x_1,x_2,x_3\}\otimes\{x_4,\ldots,x_n\} &\abb \{x_1,\ldots,x_n\}
	\end{align*}
is surjective and for every $n\geq 0$ the homomorphism $\KM_n\O \nach \iKM_n\O$ is also surjective by Proposition \ref{iKM_properties}~(8). By the commutative diagram
		\[ \begin{xy} 
		\xymatrix{
		\KM_3\O \otimes_\Z \KM_{n-3}\O \ar@{->>}[r] \ar@{->>}[d]
		& \KM_n\O \ar@{->>}[d]
		\\
		\iKM_3\O \otimes_\Z \iKM_{n-3}\O \ar[r]
		& \iKM_n\O,
		}  
		\end{xy} \]
we see that the lower map is also surjective and hence its target is divisible if $\iKM_3\O$ is divisible. Thus it suffices to show that multiplication with $\ell$ on $\iKM_3\O$ is surjective for every prime number $\ell$. Equivalently, we can show that $(\iKM_3\O)/\ell=0$ for every prime $\ell$. The case $\ell\neq p$ is the assertion of Corollary~\ref{divisible_coprime}. All in all, it remains to show that $(\iKM_3\O)/p=0$.

\vquad\noindent According to Proposition~\ref{gersten_conjecture} we have an exact sequence
	\begin{align*} 
	0 \Nach \iKM_2\O \Nach[\iota_*] \iKM_2F \Nach[\iMtame] \iKM_1\kappa \Nach 0.
	\end{align*}
Tensoring with $\Z/p$ leads to an exact sequence
	\begin{align*}
	\underbrace{\Tor_1^\Z(\Z/p,\iKM_1\kappa)}_{=0} \Nach (\iKM_2\O)/p \Nach (\iKM_2F)/p \Nach \underbrace{\iKM_1\kappa/p}_{=0},
	\end{align*}
as  $\iKM_1\kappa \cong \kappa^\times$ is a finite cyclic group with order prime to $p$.
By Proposition~\ref{delta}, we get accordingly 
	\begin{align} \label{identification}
		(\iKM_2\O)/p \cong (\iKM_2F)/p \cong 
		\begin{cases} \Z /p & \text{(if }F\text{ contains the }p\text{-th roots of unity)}
		 \\ 0& \text{(else)} \end{cases}
	\end{align}
If this is zero, we are done. Thus we consider only the other case.

\vquad\noindent Now we suppose for a while that $-1$ has a p-th root $\sqrt[p]{-1}$ in $F$. In fact, this is always the case for $p\neq 2$ or if $F$ is a finite extension of $\F_2((t))$. Let $\{u,v,w\}\in\iKM_3\O$ and let squared brackets denote residue classes modulo $p$. If $[u,x]=0$ in $(\iKM_2\O)/p\cong\Z/p$ for all $x\in\O^\times$, then
	\begin{align*}
		[u,v,w] = [u,v]\cdot[w] = 0 \in (\iKM_3\O)/p.
	\end{align*}
Otherwise, there is an $x\in\O^\times$ such that $[u,x]$ is a generator of $(\iKM_2\O)/p$. Let $[v,w]=k\cdot[u,x]$ for a suitable $k\in\Z$. Then
	\begin{align*}
		[u,v,w] = k\cdot[u,u,x] = k\cdot[-1,u,x] = kp\cdot[\sqrt[p]{-1},u,x] = 0 \in (\iKM_3\O)/p.
	\end{align*}	
	
\vquad\noindent If, in contrast, $-1$ does not have a $p$-th root in $F$, we have $p=2$ and $F$ is a finite extension of $\Q_2$. Thus $\O$ is a local and factorial domain. We consider the finite field extension $F' := F[\sqrt{-1}]$. By \cite[ch.~1, \S 6~(ii)]{serre}, the valuation ring of $F'$ is $\O':=\O[\sqrt{-1}]$. This allows us to apply Proposition~\ref{norm}.

Due to finiteness, $F'$ is also complete. The norm map $N_{F'/F} \colon \iKM_2F'\to \iKM_2F$ is surjective due to Proposition~\ref{norm_surjective} and induces a surjective map $\tilde{N}_{F'/F} \colon (\iKM_2F')/p \surj (\iKM_2F)/p$. This induces a surjective map 
	\begin{align*}
		\tilde{N}_{\O'/\O} \, \colon \,(\iKM_2\O')/p \surj (\iKM_2\O)/p
	\end{align*} 
on the Milnor K-theory of their integers via the identification \eqref{identification}. Let $[x,y,z]\in(\iKM_3\O)/p$. Decompose $[x,y,z] = [x]\cdot[y,z]$ such that $[x]\in(\iKM_1\O)/p$ and $[y,z]\in(\iKM_2\O)/p$. We find a preimage $\alpha\in\iKM_2\O'$ of $[y,z]$ under $\tilde{N}_{\O'/\O}$. The projection formula \eqref{projection_formula} yields 
 	\begin{align*}
		\tilde{N}_{\O'/\O}\bigl(\iota_*([x])\cdot \alpha\bigr) = [x]\cdot \tilde{N}_{\O'/\O}(\alpha) = [x]\cdot[y,z]=[x,y,z],
	\end{align*}
where $\iota\colon \O\inj\O'$ is the inclusion. Thus $(\iKM_3\O')/p \to (\iKM_3\O)/p$ is onto. By the previous argument, the domain of this map is zero, hence its target, too. This concludes the proof.
\end{proof}

		\section{The isomorphism $\iKM_n\O\iso\iKM_n F$ for $n\geq 3$}
		\label{section_main_result}

Our last ingredient is the following result by \df{Sivitskii}:

\begin{thm}[{\cite[p.~562]{sivitskii}}] \label{KMF_uniquely_divisible}
Let $\O$ be a complete discrete valuation ring with quotient field $F$ and finite residue field, and let $n\geq 3$. Then $\KM_nF$ is an uncountable, uniquely divisible group.
\end{thm}

The uncountability was found by \df{Bass} and \df{Tate} \cite{bass_tate} and divisibility was shown by \df{Milnor} \cite[Ex.~1.7]{milnor_k}. The torsion-freeness was originally proved by \df{Sivitskii} \cite{sivitskii} with explicit calculations; there is also a more structural proof using the Norm Residue Theorem which was formerly known as the Bloch-Kato conjecture and which was proved by \df{Rost} and \df{Voevodsky}, for a proof consider \cite[VI.~Prop.~7.1]{weibel}.
Now we are able to prove our main result.

\begin{thm} \label{thm_main_result}
Let $\O$ be a complete discrete valuation ring with quotient field $F$ and finite residue field, and let $n\geq 3$. Then the inclusion $\iota\,\colon\,\O\inj F$ induces an isomorphism
	\begin{align*}
		\iota_* \,\colon \, \iKM_n\O \iso \iKM_n F
	\end{align*}
on improved Milnor K-theory.
\end{thm}
\begin{proof}
\noindent\textbf{Surjectivity.}
Let $\kappa$ be the residue field of $\O$. According to Proposition~\ref{tame_symol_exact_sequence} we have an exact sequence
	\begin{align*}
		\iKM_n\O \Nach[\iota_*] \iKM_nF \Nach[\iMtame] \iKM_{n-1}\kappa \Nach 0.
	\end{align*}
Because $\kappa$ is a finite field and $n-1\geq 2$, we see that $\iKM_{n-1}\kappa \cong \KM_{n-1}\kappa = 0$. Thus $\iota_*$ is surjective.

\vquad\noindent\textbf{Injectivity.}
Consider the commutative diagram 
	\begin{align} \label{main_result_diagram}
	\begin{xy} 
	\xymatrix{
		\iKM_n\O \ar[rr]^{\iota_*} \ar[d]^{\PhiMQ(\O)} \ar@(dl,ul)[dd]_{\chi_n} 
		&\hspace{0cm}
		& \iKM_nF \ar@{^{(}->}[d]_{\PhiMQ(F)} \ar@(dr,ur)[dd]^{\chi_n}
		\\
		\K_n\O \ar@{^{(}->}[rr]^{\K_n(\iota)} \ar[d]^{\PhiQM(\O)} 
		&& \K_nF \ar[d]_{\PhiQM(F)} 
		\\
		\iKM_n\O \ar[rr]^{\iota_*}
		&& \iKM_nF,		
	}
	\end{xy}
	\end{align}
where $\K_n(\iota)$ is injective due to Proposition~\ref{gersten_conjecture}.
Let $\alpha\in\ker(\iota_*)$. As $\iKM_n\O$ is divisible we find $\beta\in \iKM_n\O$ such that $\alpha=\chi_n \beta$ where $\chi_n = (-1)^{n-1} (n-1)!$. Hence
	\begin{align*}
	0 = \iota_*(\alpha) = \iota_*(\chi_n \beta) = \chi_n\iota_*(\beta).
	\end{align*}
We have $\beta\in\ker(\iota_*)$ because $\iKM_nF \cong \KM_nF$ is uniquely divisible by Theorem~\ref{KMF_uniquely_divisible}. Since $\K_n(\iota)$ is injective, we also have $\ker(\iota_*)\subseteq\ker(\PhiMQ(\O))$ by diagram chasing and the latter one is clearly a subset of $\ker(\chi_n)$. We conclude that $\alpha = \chi_n \beta = 0$. So $\iota_*$ is injective, hence an isomorphism.
\end{proof}

\begin{remark}
In fact, one could simplify this proof using only $2$-divisibility for $\iKM_3\O$. With that one could proof the theorem first for $n=3$ and gets entire divisibility of $\iKM_3\O$ for free by Theorem~\ref{KMF_uniquely_divisible}. Then, one could do the proof for higher $n$. In characteristic $p\neq 2$, this would spare the proof of Theorem~\ref{KM_n_O_divisible}.
\end{remark}

Immediately from Theorem~\ref{KMF_uniquely_divisible} (respectively its proof) we obtain the following.

\begin{cor}
Let $\O$ be a complete discrete valuation ring with finite residue field. Then:
\begin{enumerate}
	\item The canonical map $\iKM_*\O \to \K_*\O$ is injective.
 	\item $\iKM_n\O$ is uniquely divisible for $n\geq 3$.
\end{enumerate}
\end{cor}

Furthermore, we are now able to generalise the exact sequence \eqref{gersten_sequence_milnor_2}.

\begin{cor} \label{cor_gersten_conjecture}
Let $\O$ be a complete discrete valuation ring with quotient field $F$ and finite residue field $\kappa$. Then we have for each $n\geq 1$ an exact sequence
	\begin{align*}
		0 \Nach \iKM_n\O \Nach[\iota_*] \iKM_nF \Nach \iKM_{n-1}\kappa \Nach 0.
	\end{align*}
\end{cor}

\appendix
\section{The ring of rational functions}

The ring of rational functions plays a crucial role for the definition of improved Milnor K-theory. This section's content relies on \df{Kerz}'s paper \cite{kerz} and elaborates some of the proofs. We start by recalling Definition~\ref{def_rational_functions} from above.

\begin{defin} \label{def_ring_rational_functions_appendix}
Let $A$ be a commutative ring and $n \in \N$. The subset
		\begin{align*}
		S := \bigl\{ \sum_{\underline{i}\in \N^n} a_{\underline{i}}\cdot \underline{t}^{\underline{i}} \in A[t_1,\ldots,t_n] \,\Sd\, \skp{a_{\underline{i}}\sd\underline{i}\in \N^n} = A \bigr\}
		\end{align*}
is multiplicatively closed, where $\underline{t}^{\underline{i}} = t_1^{i_1}\cdot\ldots\cdot t_n^{i_n}$. Define the \df{ring of rational functions} (in $n$ variables) to be
	\begin{align*}
	A(t_1,\ldots,t_n) := S\inv A[t_1,\ldots,t_n].
	\end{align*}
We obtain maps $\iota \colon A \nach A(t)$ and $\iota_1,\iota_2 \colon A(t) \nach A(t_1,t_2)$ by mapping $t$ respectively to $t_1$ or $t_2$.
\end{defin}

\begin{remark}
For a field $F$ and $n\geq 1$ we have $F(t_1,\ldots,t_n) = \Frac(F[t_1,\ldots,t_n])$.
\end{remark}

\begin{lemma}[{cf. \cite[p.~6]{kerz}}] \label{A(t)_local_infinite}
Let $(A,\mathfrak{m},\kappa)$ be a local ring and $n\geq 1$. Then the ring $A(t_1,\ldots,t_n)$ is a local ring with maximal ideal $\mathfrak{m} A(t_1,\ldots,t_n)$ and infinite residue field $\kappa(t_1,\ldots,t_n)$.
\end{lemma}
\begin{proof}
Set $B := A[t_1,\ldots,t_n]$ and $\mathfrak{n} := \mathfrak{m} A[t_1,\ldots,t_n]$. Obviously, $\mathfrak{n}\nt B$ is a prime ideal and
	\begin{align*}
		S 	& = \bigl\{ \sum_{\underline{i}\in \N^n} a_{\underline{i}}\cdot t^{\underline{i}} \in A[t_1,\ldots,t_n] 
					\sd \skp{a_{\underline{i}}|\underline{i}\in I^n} = A \bigr\} \\
			& = \bigl\{ \sum_{\underline{i}\in \N^n} a_{\underline{i}}\cdot t^{\underline{i}} \in A[t_1,\ldots,t_n] 
					\sd \exists \underline{i}\in \N^n \colon a_{\underline{i}} \notin \mathfrak{m} \bigr\} \\
			& = B \backslash \mathfrak{n},
	\end{align*}
hence $A(t_1,\ldots,t_n) = B_\mathfrak{n}$ is local with maximal ideal $\mathfrak{n}B_\mathfrak{n}$. Its residue field is
	\begin{align*}
		B_\mathfrak{n}/\mathfrak{n}B_\mathfrak{n}
			 \cong \bigl( B/\mathfrak{n} \bigr)_\mathfrak{n} 
			 = \Frac\bigl( \kappa[t_1,\ldots,t_n]) 
			 = \kappa(t_1,\ldots,t_n).
	\end{align*}
\end{proof}

\begin{lemma} \label{lem_going_up}
Let $(A,\mathfrak{m})$ be a local ring and $A \inj B$ an integral ring extension such that $\mathfrak{m} B\nt B$ is a prime ideal. Then $B$ is also local with maximal ideal $\mathfrak{m} B$.
\end{lemma}
\begin{proof}
According to the Going-Up theorem (see \cite[pt.~I, ch.~2, Thm.~5]{matsumura}) the maximal ideals of $B$ are precisely the prime ideals of $B$ lying over $\mathfrak{m}$. But every prime ideal $\p\nt B$ with $A\cap \p \supseteq \mathfrak{m}$ must contain $\mathfrak{m} B$. Thus it is the only maximal ideal of $B$.
\end{proof}

\begin{lemma} \label{localisation_injecitvie}
Let $A$ be a ring, $S\subset A$ a multiplicatively closed subset and let $\iota \colon A \to S\inv A$ the canonical homomorphism. If $S$ does not contain any zero-divisors, then $\iota$ is injective. In this case a ring homomorphism $f_S \colon S\inv A\nach B$ is injective if and only if the composition $f \colon A \to B$ is injective.
\end{lemma}
\begin{proof}
If $f$ is injective and $0 = f_S(\frac{a}{s})=f(a)f(s)\inv$, then $f(a)=0$ since $f(s)\inv$ is a unit. The rest is omitted.
\end{proof}

\begin{prop} \label{ring_of_fractions_base_change}
Let $A \inj B$ be a flat, local (i.e.~the maximal ideal of $A$ generates the maximal ideal of $B$), and integral extension of local rings, and $n\geq 1$. Then the canonical homomorphism
	\begin{align*}
		\phi \,\colon\, \B \otimes_A A(t_1,\ldots,t_n) \iso B(t_1,\ldots,t_n)
	\end{align*}
is an isomorphism. 
In particular, the induced homomorphism $A(t_1,\ldots,t_n)\nach B(t_1,\ldots,t_n)$ is an extension of local rings with infinite residue fields.
\end{prop}
\begin{proof}
For convenience, we write ``$t$'' for ``$t_1,\ldots,t_n$''.
Let $\mathfrak{m} \nt A$ and $\mathfrak{n} \nt B$ the maximal ideals, hence $\mathfrak{n}=\mathfrak{m} B$. We set $\mathfrak{m}_t := \mathfrak{m} A[t]$ and $\mathfrak{n}_t := \mathfrak{n} B[t] = \mathfrak{m} B[t]$. 
We use some facts for ring maps corresponding to the permanences of properties of schemes as stated in \cite[Appendix~C]{gw}. 

\vquad\noindent\textbf{Injectivity.} The flat base change $A[t]\inj B[t]$ is injective. According to Lemma~\ref{localisation_injecitvie}, the localisation $B[t]\inj B(t)$ is injective. Applying again Lemma~\ref{localisation_injecitvie}, this time to $A[t] \inj A(t)$, and using the commutative diagram
		\[ \begin{xy} 
		\xymatrix{
		A[t] \ar@{^{(}->}[r] \ar@{^{(}->}[d] & B[t] \ar@{^{(}->}[d] \\
		A(t) \ar[r] & B(t),
		}  
		\end{xy} \]
we see that $A(t) \inj B(t)$ is injective, hence also $B\otimes_AA(t) \inj B\otimes_AB(t)$ since $B$ is a flat $A$-module. Thus the commutative diagram
		\[ \begin{xy} 
		\xymatrix{
		B\otimes_AA(t) \ar[rr]^\phi \ar@{^{(}->}[dr] && B(t) \ar@{^{(}->}[dl] \\
		&B\otimes_AB(t)
		}  
		\end{xy} \]
tells us that $\phi$ is injective.

\vquad\noindent\textbf{Surjectivity.} The base change  $A(t)\inj B\otimes_AA(t)$ is an integral extension. By Lemma~\ref{lem_going_up}, $B\otimes_AA(t)$ is a local ring with maximal ideal $\mathfrak{m}_t(B\otimes_AA(t))$. Furthermore, we have a commutative diagram
		\[ \begin{xy} 
		\xymatrix{
		B\otimes_AA(t) \ar[r]^-{\cong} \ar[d]_\phi
		& B\otimes_A \bigl( A[t]\otimes_{A[t]} A(t) \bigr) \ar[r]^{\cong} 
		& \bigl( B\otimes_A A[t] \bigr) \otimes_{A[t]} A(t) \ar[d]^-\cong 
		\\ B(t) && B[t] \otimes_{A[t]} A(t) \ar[ll]_{\phi'}
		}  
		\end{xy} \]
Hence surjectivity for $\phi$ is equivalent to surjectivity for $\phi'$.
Consider the homomorphism
	\begin{align*}
		\psi \colon B[t] &\Nach B[t]\otimes_{A[t]}A(t) \\
		f &\abb f\otimes 1
	\end{align*}
If $f = \sum_{i=0}^d b_it^i \in B[t]\backslash\mathfrak{n}_t$, then there is a $j$ such that $b_j$ is a unit in $B$. Then its image $\psi(b_j)$ is a unit as well and hence $\psi(f) \in\bigl(B[t]\otimes_{A[t]}A(t)\bigr)^\times$.
Thus the universal property of localisation yields a homomorphism
	\begin{align*}
		\psi' \colon B(t) \Nach B[t]\otimes_{A[t]}A(t)
	\end{align*}
such that $\psi = \psi'\circ\iota$ where $\iota \colon B[t]\inj B(t)$. The commutative diagram
		\[ \begin{xy} 
		\xymatrix{
		B[t] \ar[rr]^\iota \ar[dd]_{\id_{B[t]}} \ar[dr]^\psi  
		&& B(t) \ar[dl]_{\psi'} \ar@{.>}[dd]^{\id_{B(t)}} \\
		& B[t]\otimes_{A[t]}A(t)  \ar[dr]^{\phi'} \\
		B[t] \ar[rr]^\iota && B(t)
		}  
		\end{xy} \]
and the uniqueness for lifts to the localisation show $\phi'\circ\psi'=\id_{B(t)}$. So $\phi'$ is surjective.
\end{proof}

	\section{Milnor K-theory and algebraic K-theory}

This section is dedicated to \df{Kerz}'s Theorem~\ref{MQM} which we restate at this place again. 

\begin{thm}[{\cite[Prop.~10~(6)]{kerz}}] \label{MQM_appendix}
Let $A$ be a local ring and $n\in\N$. Then there exists a natural map
	\begin{align*}
		\PhiMQ(A) \,\colon\, \iKM_nA \Nach \K_nA
	\end{align*}
as well as a natural map
	\begin{align*}
		\PhiQM(A) \,\colon\, \K_nA \Nach \iKM_nA
	\end{align*}
such that the composition
	\begin{align*}
		\iKM_nA \Nach[\PhiMQ(A)] \K_nA \Nach[\PhiQM(A)] \iKM_nA
	\end{align*}
is multiplication with $\chi_n := (-1)^{n-1}\cdot (n-1)!$.
\end{thm}

As there are few details given in \cite{kerz}, we present more of them in this section. The statement relies on an analogous statement by \df{Nesterenko} and \df{Suslin} for local rings with infinite residue field and classical Milnor K-theory. 

\vquad\noindent For any commutative ring $A$ there exists a graded-commutative product map $K_1(A) \otimes\ldots\otimes K_1(A) \nach \K_n(A)$ such that the Steinberg relations are satisfied \cite[IV.1.10.1]{weibel}. Thus there exists a graded map
	\begin{align*}
		\PhiMQ'(A) \,\colon\, \KM_*A \Nach \K_*A.
	\end{align*}
In the case of a local ring with an infinite residue field, there is also a map in the other direction.

\begin{thm}[{\cite[Thm.~4.1]{suslin}}] \label{suslin_appendix_thm}
Let $A$ be a local ring with infinite residue field and $n\in\N$. Then there exists a natural map
	\begin{align*}
		\PhiQM'(A) \,\colon\, \K_nA \Nach \KM_n(A)
	\end{align*}
such that the composition
	\begin{align*}
		\KM_nA \Nach[\PhiMQ'(A)] \K_nA \Nach[\PhiQM'(A)] \KM_nA
	\end{align*}
is multiplication with $\chi_n = (-1)^{n-1}\cdot (n-1)!$.
\end{thm}

We say a few words on the construction of the map $\PhiQM'(A)$. By \cite[Thm.~3.25]{suslin} there are natural isomorphisms
	\begin{align*}
		\Ho_n(\GL_n(A)) &\iso \Ho_n(\GL(A)) \text{  and} \\
		\Ho_n(\GL_n(A))/\Ho_n(\GL_{n-1}(A)) &\iso \KM_nA.
	\end{align*}
for a local ring $A$ with infinite residue field. So a homomorphism $\PhiQM'(A)$ can be defined as the composition of homomorphisms such that the following diagram commutes.
	\small
	\[ \begin{xy} 
	\xymatrix{
		\K_nA \ar@{.>}[d]^{\PhiQM'(A)} & \ar@{=}[l] \pi_n(\BGL(A)^+) \ar[r]^-{\text{Hurewicz}}
		&  \Ho_n(\BGL(A)^+) \ar@{=}[r]& \Ho_n(\GL(A)) \\
		\KM_nA && \ar[ll]_-\cong \Ho_n(\GL_n(A))/\Ho_n(\GL_{n-1}(A))  
		& \Ho_n(\GL_n(A)) \ar[l] \ar[u]^-\cong
	}  
	\end{xy} \]
	\normalsize
	
\vquad\noindent For passing from the case of local rings with infinite residue fields to the case of arbitrary local rings, we consider the construction of improved Milnor K-theory in a more general setting. Given a functor $E$ from rings to abelian groups, we can associate to it an \emph{improved} functor $\hat{E}$ given by
	\begin{align*}
		\hat{E}(A) := \ker\bigl[ E(A(t)) \Nach[\delta] E(A(t_1,t_2)) \bigr]
	\end{align*}
where $A(t)$ and $A(t_1,t_2)$ are rings of rational functions  and $\delta$ is the map $E(\iota_1)-E(\iota_2)$ (cf.~Definition~\ref{def_ring_rational_functions_appendix}). This construction is essentially due to \df{Gabber} \cite{gabber} and was investigated by \df{Kerz} \cite{kerz}. The latter one proved \cite[Prop.~9]{kerz} that the canonical map $E(A) \nach \hat{E}(A)$ is an isomorphism if one of the following two conditions holds
	\begin{enumerate}
		\item The ring $A$ is a local with infinite residue field.
		\item The ring $A$ is local and $E$ admits compatible norm maps for all finite \etale{} 	extensions of local rings (cf. \cite[p.~5]{kerz}).
	\end{enumerate}	 

Algebraic K-theory admits those norm maps (cf.~\cite[IV.6.3.2]{weibel} where they are called ``transfer maps''). Thus for any local ring $A$ we have a natural isomorphism 
	\begin{align} \label{quillen_hat_iso}
		\K_*A \Nach[\cong] \hat{\K}_*A.
	\end{align}
	
\begin{proof}[Proof of Theorem~\ref{MQM_appendix}]
Let $A$ be a local ring. Then the rings $A(t)$ and $A(t_1,t_2)$ are local rings with infinite residue field, see Lemma~\ref{A(t)_local_infinite}. By naturality of the maps $\PhiMQ'$ and $\PhiQM'$ together with the isomorphism~\eqref{quillen_hat_iso}, we obtain the desired maps via the universal property of the kernel as indicated in the diagram
	\[ \begin{xy} 
	\xymatrix{
		& \iKM_*A \ar@{^{(}->}[r]^{\text{kernel}}  \ar@{.>}[d]_{\PhiMQ(A)} 
		& \KM_*A(t) \ar[r]^{\delta^\mathrm{M}_*} \ar[d]^{\PhiMQ'(A(t))}
		& \KM_*A(t_1,t_2) \ar[d]^{\PhiMQ'(A(t_1,t_2))}
		\\
		& \K_*A \ar@{.>}[d]_{\PhiQM(A)} \ar@{^{(}->}[r]^{\text{kernel}} 
		& \K_*A(t) \ar[d]^{\PhiQM'(A(t))} \ar[r]^{\delta^\mathrm{Q}_*}
		& \K_*A(t_1,t_2) \ar[d]^{\PhiQM'(A(t_1,t_2))} 
		\\
		& \iKM_*A \ar@{^{(}->}[r]^{\text{kernel}}
		& \KM_*A(t) \ar[r]^{\delta^\mathrm{M}_*}
		& \KM_*A(t_1,t_2).
	}  
	\end{xy} \] 
By Theorem~\ref{suslin_appendix_thm}, the middle vertical composition is multiplication with the natural number $\chi_n$, hence this also holds for the left vertical composition.
\end{proof}

	\section{The norm map}
There is a norm map for Milnor K-theory of fields. More precisely, for every field extension $F \inj F'$ there is a homomorphism $N_{F'/F} \colon \KM_*F' \nach \KM_*F$ of graded $\KM_*F$-modules such that the following two conditions hold.
	\begin{enumerate}
		\item \textbf{Functoriality}. For every tower $F\inj F'\inj F''$ of fields holds $N_{F/F}=\id$ and $N_{F'/F}\circ N_{F''/F'}=N_{F''/F}$.
		\item \textbf{Reciprocity.} For $\alpha\in\KM_*F(t)$ holds
			\begin{align*}
				\sum_v N_{\kappa(v)/F}\circ\partial_v(\alpha) = 0
			\end{align*}
		where $v$ runs over all discrete valuations of $F(t)$ over $F$ and $\kappa(v)$ is the residue field of $v$.
	\end{enumerate}
We give a brief sketch of the construction of the norm map. For a detailed exposition we refer the reader to \cite[\S~7.3]{gille}.	
	
\vquad \noindent	
For a finite field extension $F'/F$ such that $F'\cong F[X]/\skp{\pi}$ where $\pi$ is monic and irreducible, the norm map is defined via the split exact Bass-Tate sequence
	\begin{align} \label{bass_tate_sequence}
	0 \Nach \KM_nF \Nach \KM_nF(X) \overset{\bigoplus \partial_P}{\Nach} \bigoplus_P \KM_{n-1}(F[X]/\skp{P}) \Nach 0,
	\end{align}
where the sum is over all monic irreducible $P\in F[X]$ and $\bigoplus \partial_P$ is the well-defined sum of the tame symbols $\partial_P$ (with respect to the discrete valuation which is associated to the prime element $P$, see Theorem~\ref{tame_symbol}). Taking coproducts over all values of $n$, we obtain an exact sequence of $\KM_*F$-modules.  Let $\partial_\infty$ be the tame symbol corresponding to the negative degree valuation on $F(X)$. Its residue field is isomorphic to $F$. The composition $\KM_*F \nach \KM_*F(X) \nach[\partial_\infty] \KM_{*-1}F$ vanishes since all elements in $F$ have degree zero. By the universal property of the cokernel, we obtain a map $N$ as indicated in the diagram (where all the maps have degree zero).
	\begin{align} \label{eq_norm_classical}
	\begin{xy} 
	\xymatrix{
		0 \ar[r]
		& \KM_*F \ar[r]^-{\iota_*} \ar[rd]_-0
		& \KM_*F(X) \ar[r]^-{\bigoplus \partial_P} \ar[d]^-{\partial_\infty}
		& \bigoplus_P \KM_{*-1}(F[X]/\skp{P}) \ar[r] \ar@{.>}[dl]^-N
		& 0
		\\
		&& \KM_{*-1}F 
		}
	\end{xy}
	\end{align}
By precomposing $N$ with the inclusion of the factor $\K_{*-1}(F[X]/\skp{\pi})$, we get the \df{norm map} 
	\begin{align*}
		N_{F'/F} \colon \KM_*F' \nach \KM_*F
	\end{align*}	 
of the extension $F'/F$.

All the maps in diagram \eqref{eq_norm_classical} are (graded) $\KM_*F$-linear. This is clear for the map $\iota_*$ which is induced by the inclusion $\iota \colon F \nach F(X)$. With the explicit description of the tame symbol (see Theorem~\ref{tame_symbol}) we deduce linearity for the tame symbols $\partial : \KM_*F(X) \nach \KM_{*-1}\kappa$ where $\kappa=F$ (if $\partial=\partial_\infty$) or $\kappa=F[X]/\skp{P}$ (if $\partial=\partial_P$ for a monic irreducible $P\in F[X]$). Let $\pi$ be a uniformiser and $\O$ the valuation ring with respect to the valuation in question. Then $\KM_*F(X)$ is additively generated by the set
	\begin{align*}
		\Bigl\{ \{\pi,u_2,\ldots,u_n\},\{u_1,\ldots,u_n\} \Sd n\geq 1, u_1,\ldots,u_n \in \O^\times \Bigr\}.
	\end{align*}		
We note that $\partial(\{u_1,\ldots,u_n\})=0$ for all $u_1,\ldots,u_n\in\O^\times$, hence linearity for those elements is clear. Now let $x_1,\ldots, x_k\in F^\times$. Observe that $F^\times \subseteq \O^\times$ and that the map $F\nach\kappa$ is injective. Thus we see $\KM_*F$-linearity as follows.
	\begin{align} \label{eq_tame_linearity}
		\partial \bigl( \{x_1, \ldots, x_k\} \cdot \{\pi,u_2,\ldots,u_n\} \bigr)
			&= \partial \bigl( \{x_1, \ldots, x_k, \pi,u_2,\ldots,u_n\} \bigr) \\
			&= \partial \bigl( (-1)^k \{\pi, x_1, \ldots, x_k,u_2,\ldots,u_n\} \bigr) \nonumber \\
			&= (-1)^k\{\bar{x}_1, \ldots, \bar{x}_k,\bar{u}_2,\ldots,\bar{u}_n\} \nonumber \\
			&= (-1)^k\{x_1, \ldots, x_k\} \cdot \{\bar{u}_2,\ldots,\bar{u}_n\}  \nonumber \\
			&= (-1)^k\{x_1, \ldots, x_k\} \cdot  \partial \bigl( \{\pi,u_2,\ldots,u_n\} \bigr) \nonumber
	\end{align}
The sign appears for $\partial$ having degree $-1$. Hence $\partial_\infty$ and also the direct sum $\bigoplus \partial_P$ in \eqref{eq_norm_classical} are $\KM_*F$-linear. Thus this also holds for the induced map $N$ and (since the inclusion of a factor is also linear) as well for the norm map $N_{F'/F}$.  

For an arbitrary finite field extension $F(\alpha_1,\ldots,\alpha_n)/F$, we define the norm map via a decomposition
	\begin{align*}
		F \subset F(\alpha_1) \subset F(\alpha_1,\alpha_2) \subset \ldots \subset F(\alpha_1,\ldots,\alpha_n).
	\end{align*}
This construction is due to \df{Bass} and \df{Tate} \cite[\S 5]{bass_tate}. Its independence of the generating family $(\alpha_1,\ldots,\alpha_n)$ was proven by \df{Kato} \cite{kato}. 

\vquad\noindent
\df{Kerz} extended this to the realm of finite \etale{} extensions of semi-local rings with infinite residue fields \cite{kerz_gersten_conjecture}. For improved Milnor K-theory, he also showed the existence of norm maps for finite \etale{} extensions of arbitrary local rings by reducing to the case of infinite residue fields and classical Milnor K-theory \cite{kerz}.

As a matter of fact, finite \etale{} extensions $A\inj B$ of local rings are precisely those of the form $B \cong A[t]/\skp{\pi}$ with $\pi$ monic irreducible and $\mathrm{Disc}(\pi)\in A^\times$. Our aim is to drop the condition with the discriminant. This is possible by restricting to factorial domains so that we can use a Bass-Tate-like sequence by \df{Kerz}. This is basically \df{Kerz}' work though not stated literally by himself.

\begin{prop} \label{norm_appendix}
Let $A$ be a local and factorial domain and $\iota\colon A\inj B$ an extension of local rings such that $B\cong A[X]/\skp{\pi}$ for a monic irreducible polynomial $\pi\in A[X]$.
Then there exists a norm map
	\begin{align*}
		N_{B/A} \,\colon\, \iKM_*B \Nach \iKM_*A
	\end{align*}
satisfying the projection formula
	\begin{align} \label{projection_formula_appendix}
		N_{B/A}\bigl(\{\iota_*(x),y\}\bigr) = \bigl\{x,N_{B/A}(y)\bigr\}
	\end{align}
for all $x\in\iKM_*A$ and $y\in\iKM_*B$. 
\end{prop}
\begin{proof} 
References within this proof refer to \cite[\S 4]{kerz_gersten_conjecture} unless said otherwise. For semi-local domains with infinite residue field, there is a split exact sequence
	\begin{align} \label{bass-tate-kerz}
	0 \Nach \KM_nA \Nach \K^\mathrm{t}_nA \Nach[\bigoplus \partial] \bigoplus_P \KM_{n-1}(A[X]/\skp{P}) \Nach 0,
	\end{align}
where the sum is over all monic irreducible $P\in A[X]$, see [Thm.~4.4]. Here $\K^\mathrm{t}_nA$ is an appropriate group generated by so-called \emph{feasible} tuples of elements in $\Frac(A)(X)$. A tuple
	\begin{align*}
		\Bigl(\frac{p_1}{q_1},\ldots,\frac{p_n}{q_n}\Bigr)
	\end{align*}
with $p_i,q_i\in A[X]$ and all $p_i/q_i$ in reduced form is \df{feasible} iff the highest nonvanishing coefficients of $p_i,q_i$ are invertible in $A$ and for irreducible factors $u$ of $p_i$ or $q_i$ and $v$ of $p_j$ or $q_j$ ($i\neq j$), we have $u=av$ with $a\in A^\times$ or $(u,v)=1$, see [Def.~4.1]. Now $\K^\mathrm{t}_nA$ is defined by modding out relations for $A^\times$-linearity which yield a $\KM_*A$-module structure on $\K^\mathrm{t}_*A$, and the relations
	\begin{align*}
		(p_1,\ldots,p,1-p,\ldots,p_n)=0 \quad\text{and}\quad (p_1,\ldots,p,-p,\ldots,p_n)=0
	\end{align*}
for $p\in\Frac(A)(X)$, see [Def.~4.2].

The map $\bigoplus \partial_P$ is similar to the tame symbol appearing in \eqref{bass_tate_sequence}. First, for all monic irreducible $P\in A[t]$ there are maps $\partial_P\colon\K^\mathrm{t}_nA\nach\KM_{n-1}A[t]/\skp{P}$ satisfying
	\begin{align} \label{eq_improved_tame_formula}
		\partial_P \{P,p_2,\ldots,p_n\}) = \{\overline{p_2},\ldots,\overline{p_n}\}
	\end{align}
for $p_i\in A[X]$ such that $(P,p_i)=1$ [Proof of Theorem.~4.4, Step~2]. Then $\bigoplus \partial$ is the well-defined sum of those $\partial_P$. Analogously, we define a map $\partial_\infty\colon\K^\mathrm{t}_nA\nach\KM_{n-1}A$ corresponding to the negative degree valuation by using $X\inv$ as a uniformiser.

\vquad\noindent
With this we can define norm maps in the case of an infinite residue field exactly as in the case of fields via a map $N$ as indicated in the following diagram.
	\[ \begin{xy}
		\xymatrix{
		0 \ar[r]
		& \KM_*A \ar[r]^i \ar[rd]_-0
		& \K^\mathrm{t}_*A \ar[r]^-{\bigoplus \partial} \ar[d]^-{\partial_\infty}
		& \bigoplus_P \KM_*(A[X]/\skp{P}) \ar[r] \ar@{.>}[dl]^-N
		& 0.
		\\
		&& \KM_*A 
		}
	\end{xy} \]
Again, all the maps are $\KM_*A$-linear. For the map $i$ this follows directly from the construction above. For the tame symbols this holds as in \eqref{eq_tame_linearity} since they satisfy \eqref{eq_improved_tame_formula}. For the map $N$ this is a formal consequence. Again, we obtain $\KM_*A$-linear norm map $N_{B/A} \colon \KM_*B \nach \KM_*A$ for $B=A[X]/\skp{\pi}$ by precomposing with the linear inclusion of the factor. The projection formula \eqref{projection_formula_appendix} is nothing else than the statement that the norm maps are $\KM_*A$-linear.

\vquad\noindent
Now we treat the case of arbitrary residue fields. Let $A$ be an arbitrary factorial and local domain and $A\inj B$ an extension of local rings such that $B\cong A[X]/\skp{\pi}$ for a monic irreducible polynomial $\pi\in A[X]$.

\vquad
\textbf{Claim.} $(A[X]/\skp{\pi})(t_1,\ldots,t_k) \cong A(t_1,\ldots,t_k)[X]/\skp{\pi}$.

\vquad\noindent To see this, we first observe that the right-hand side is isomorphic to the tensor product $A(t_1,\ldots,t_k)\otimes_AA[X]/\skp{\pi}$. By Proposition~\ref{ring_of_fractions_base_change} above, the left-hand side fulfils this as well.

If $A$ has residue field $\kappa$, then $A(t_1,\ldots,t_k)$ has residue field $\kappa(t_1,\ldots,t_k)$ by Lemma~\ref{A(t)_local_infinite}. Hence we can reduce to the situation of infinite residue fields. The diagram
	\[ \begin{xy} 
		\xymatrix{		
		\KM_nA(t) \ar@{^{(}->}[r] \ar[d]^\delta
		& \K^\mathrm{t}_nA(t) \ar@{->>}[r]^-{\partial} \ar[d]^\delta
		& \bigoplus_P \KM_{n-1}(A(t)[X]/\skp{P}) \ar[d]^\delta
		\\	
		\KM_nA(t_1,t_2) \ar@{^{(}->}[r]
		& \K^\mathrm{t}_nA(t_1,t_2) \ar@{->>}[r]^-{\partial} 
		& \bigoplus_Q \KM_{n-1}(A(t_1,t_2)[X]/\skp{Q}) 
	}  
	\end{xy} \]	
commutes where $\delta = (\iota_1)_*-(\iota_2)_*$ and $P$ and $Q$ run over all monic irreducible polynomials over $A(t)$ respectively over $A(t_1,t_2)$ (note that a monic irreducible polynomial over $A(t)$ is also monic irreducible over $A(t_1,t_2)$ both via $\iota_1$ and via $\iota_2$). Thus the right-hand square in the diagram
	\[ \begin{xy} 
		\xymatrix{
		\iKM_nB \ar@{^{(}->}[r] \ar@{.>}[d]^{N_{B/A}}
		& \KM_nB(t) \ar[r]^-{\delta_n^\mathrm{M}} \ar[d]^{N_{B(t)/A(t)}}
		& \KM_nB(t_1,t_2) \ar[d]^{N_{B(t_1,t_2)/A(t_1,t_2)}}
		\\
		\iKM_nA \ar@{^{(}->}[r]
		& \KM_nA(t) \ar[r]^-{\delta_n^\mathrm{M}}
		& \KM_nA(t_1,t_2)
	}  
	\end{xy} \]	
commutes which enables us to define the desired norm map $N_{B/A}$ via restriction and the projection formula inherits from the case of infinite residue fields to the case of arbitrary residue fields.
\end{proof}


\end{document}